\documentclass{birkjour}
 
 \usepackage{amsopn}
 \usepackage{amsmath,amsthm,amssymb}

 \newcommand{\nc}{\newcommand}

\nc{\bb}{\mathfrak{b} }
 \nc{\cc}{\mathfrak{c} }  \nc{\dd}{\mathfrak{d} } 
    \nc{\ggo}{\mathfrak{g} }
 \nc{\hh}{\mathfrak{h} }  \nc{\ii}{\mathfrak{i} }
 \nc{\jj}{\mathfrak{j} }  \nc{\kk}{\mathfrak{k} }
\nc{\mm}{\mathfrak{m} }   \nc{\nn}{\mathfrak{n} }
\nc{\pp}{\mathfrak{p} }   
\nc{\rr}{\mathfrak{r} } \nc{\sg}{\mathfrak{s} }
 \nc{\sso}{\mathfrak{so} }  \nc{\spg}{\mathfrak{sp} }
 \nc{\ssu}{\mathfrak{su} }  \nc{\ssl}{\mathfrak{sl} }
 \nc{\tog}{\mathfrak{t} }  \nc{\uu}{\mathfrak{u} }
 \nc{\vv}{\mathfrak{v} } \nc{\ww}{\mathfrak{w} }
 \nc{\zz}{\mathfrak{z} }

\nc{\CC}{{\mathbb C}}
 \nc{\DD}{{\mathbb D}}
\nc{\FF}{{\mathbb F}}
\nc{\GG}{{\mathbb G}}  
\nc{\HH}{{\mathbb H}}
\nc{\II}{{\mathbb I}}
\nc{\JJ}{{\mathbb J}}
\nc{\KK}{{\mathbb K}}
\nc{\NN}{{\mathbb N}}

\nc{\RR}{{\mathbb R}}  
 \nc{\ZZ}{{\mathbb Z}}  
 
 \newcommand{\Heis}{\mathrm{H}}

\nc{\ggob}{\overline{\mathfrak{g}}} 
 
\nc{\glg}{\mathfrak{gl} }
  
\nc{\pca}{\mathcal{P}} \nc{\nca}{\mathcal{N}}
 
 \nc{\vp}{\varphi} \nc{\ddt}{\frac{{\rm d}}{{\rm d}t}}
 \nc{\la}{\langle} \nc{\ra}{\rangle}
 \nc{\brg}{[\,,\,]_{\ggo}}
 \nc{\brv}{[\,,\,]_{\vv}}
 
 \nc{\SO}{{\sf SO}} \nc{\Spe}{{\sf Sp}} \nc{\Sl}{{\sf Sl}}
 \nc{\SU}{{\sf SU}} \nc{\Or}{{\sf O}} \nc{\U}{{\sf U}}
 \nc{\Gl}{{\sf Gl}} \nc{\Se}{{\sf S}} \nc{\Cl}{{\sf Cl}}
 \nc{\Spin}{{\sf Spin}} \nc{\Pin}{{\sf Pin}}
 \nc{\Id}{{\sf Id}} \nc{\Is}{{\sf I}}
 \nc{\Rc}{{\sf Rc}} 
 
 \nc{\ad}{\operatorname{ad}} \nc{\Ad}{\operatorname{Ad}}
 \nc{\coad}{\operatorname{coad}} 
 \nc{\rank}{\operatorname{rank}} \nc{\Irr}{\operatorname{Irr}}
 \nc{\End}{\operatorname{End}} \nc{\Aut}{\operatorname{Aut}}
 \nc{\Inn}{\operatorname{Inn}} \nc{\Der}{\operatorname{Der}}
 \nc{\Ker}{\operatorname{Ker}} \nc{\Iso}{\operatorname{I}}
 \nc{\Le}{\operatorname{L}} \nc{\tr}{\operatorname{tr}}
 \nc{\dif}{\operatorname{d}} \nc{\sen}{\operatorname{sen}}
 \nc{\modu}{\operatorname{mod}} \nc{\Ric}{\operatorname{R}}
 \nc{\Sym}{\operatorname{Sym}} \nc{\sca}{\operatorname{sc}}
 \nc{\scalar}{{\sf s}} \nc{\grad}{\operatorname{grad}}
 \nc{\ricci}{\operatorname{r}} \nc{\riccin}{\operatorname{Ric}}
 \nc{\Lie}{\operatorname{L}} \nc{\ct}{\operatorname{T}}

\newcommand{\deax}{\partial_x}
\newcommand{\deay}{\partial_y}
\newcommand{\deaz}{\partial_z}
\newcommand{\deat}{\partial_t}

\nc{\mr}{{\mathfrak r}}
\nc{\ms}{{\mathfrak s}}
\nc{\mv}{{\mathfrak v}}
\nc{\lra}{\longrightarrow}
\nc{\R}{{\mathbb R}}
\nc{\Z}{{\mathbb Z}}

 \theoremstyle{plain}
 \newtheorem{thm}{Theorem}[section]
 \newtheorem{prop}[thm]{Proposition}
 
 \newtheorem{lem}[thm]{Lemma}
 
 \theoremstyle{definition}

 \theoremstyle{remark}
 \newtheorem{rem}{Remark}
 
 \newtheorem{exa}[thm]{Example}

 \newcommand{\ri}{{\rm (i)}}
 \newcommand{\rii}{{\rm (ii)}}

\begin{document}
\title[Naturally reductive pseudo-Riemannian Lie groups in low dimensions]
{Naturally reductive pseudo-Riemannian Lie groups in low dimensions}

\author{V. del Barco}
\email{delbarc@fceia.unr.edu.ar}
\author{G. P. Ovando}
\email{gabriela@fceia.unr.edu.ar}
\author{F. Vittone}
\email{vittone@fceia.unr.edu.ar}

\address{V. del Barco, G. P. Ovando, F. Vittone: Depto de Matem\'atica, ECEN-FCEIA, Universidad Nacional de Rosario \\Pellegrini 250, 2000 Rosario, Santa Fe, Argentina.}

\date{\today}

\thanks{V. del Barco and F. Vittone: Universidad Nacional de Rosario. Supported by CONICET fellowship.}
\thanks{G. Ovando: CONICET and Universidad Nacional de Rosario.}

\thanks{Work partially supported by SCyT-U. N. Rosario, ANPCyT, CONICET.}
\begin{abstract} This work concerns  the non-flat metrics on the Heisenberg Lie group of dimension 
three $\Heis_3(\RR)$ and the bi-invariant metrics on the  solvable Lie groups of dimension four. On $\Heis_3(\RR)$ we prove that 
the property of the metric  being  naturally reductive is equivalent to the property of the center  being non-degenerate. These metrics  are Lorentzian algebraic Ricci solitons.  We start with the indecomposable Lie groups of
dimension four  admitting  bi-invariant metrics and which act on $\Heis_3(\RR)$ by isometries and we finally  study some geometrical features on these spaces. 
 
\end{abstract}

\subjclass{ 53C50 53C30 22E25 57S25. }

\keywords{Pseudo-Riemannian spaces, naturally reductive, Lie groups, Heisenberg group}

\date{ today }

\maketitle

\section{Introduction}
Homogeneous manifolds constitute the goal of several modern research in pseudo-Riemannian geometry, 
for instance  Lorentzian spaces for which all null geodesics are homogeneous became relevant in physics \cite{FMP,Me}. This fact motivated several studies on g.o. spaces 
in the last years, see for instance \cite{Ca, CM1, CM2, Du} and its references. In particular  a three-dimensional connected, simply connected, complete homogeneous Lorentzian manifold  is symmetric, or it is isometric to a
three-dimensional Lie group equipped with a left-invariant Lorentzian metric \cite{Ca}.

In the case of the Heisenberg Lie group of dimension three $\Heis_3(\RR)$ it was proved 
in \cite{Ra}  that there are three classes of left-invariant Lorentzian metrics, and only  one of them 
is flat (see also \cite{No}), which is  characterized by the property of the center  being degenerate. 

In this work we concentrate the attention to the other two non-flat metrics on $\Heis_3(\RR)$
 and their isometry groups. According to \cite{Ov2} any left-invariant metric on a
 Heisenberg Lie
group, for which the center is non-degenerate is 
naturally reductive, so these spaces are geodesically 
complete and  non-flat. Here we prove a partial converse to that result: {\em Any naturally reductive 
Lorentzian metric on $\Heis_3(\RR)$  admitting an action by isometric isomorphisms of a
 one-dimensional group,  
 restricts to a metric on the center. }

Thus for any left-invariant Lorentzian metric on  $\Heis_3(\RR)$  the following statements are equivalent:
\begin{itemize} \label{equiv}
\item non-flat metric,
\item non-degenerate center,
\item naturally reductive metric.
\end{itemize}

The first equivalences  follow from Theorem 1 in \cite{Ge}. The statement  above  does not hold in higher dimensions: a flat left-invariant 
Lorentzian metric on  $\RR\times \Heis_3(\RR)$ is proved to be naturally reductive in \cite{Ov3}.
Properties of flat or Ricci-flat Lorentzian metrics were investigated for instance in \cite{ABL, Bo, Ge} and references therein.
Here we also compute the corresponding isometry groups following results on naturally reductive metrics in \cite{Ov2} (comparing with \cite{BR})  and we see that the non-flat metrics are algebraic Ricci solitons (see \cite{BO}).

The study of these  naturally reductive non-flat metrics on $\Heis_3(\RR)$ is motivated by 
 the results on \cite{Ov1}, which state that a  naturally reductive pseudo-Riemannian space admits a transitive action by isometries of a Lie group equipped with a  bi-invariant metric.  Hence we start with the classification of all Lie algebras up to dimension four admitting an ad-invariant metric. It is important to remark that the method used here is constructive an independent of the classification of low dimensional Lie algebras.  

So a naturally reductive Lorentzian metric on  $\Heis_3(\RR)$ admits an action by isometries of a Lie group $G$ with a bi-invariant metric. If $G$ has dimension four, it corresponds to one of the Lie algebras obtained before. This is a key point in the proof of the  equivalence stated above.  

 Finally we complete the work by  investigating the geometry of the bi-invariant metrics of the solvable Lie groups  $G_0$ and $G_1$, which are associated to the non-flat metrics on $\Heis_3(\RR)$. We compute the isometry groups $\Is(G_0)$ and $\Is(G_1)$ in the aim of establishing a relationship between them and $G_0$ and $G_1$ as isometry groups of  $\Heis_3(\RR)$. Also geodesics are described.   

\section{Lie algebras with ad-invariant metrics up to dimension four} 

In this section we revisit the Lie algebras of dimension $d\leq 4$ that can be furnished
 with an ad-invariant metric. The proofs given here are constructive and they do not make use 
of  the double extension procedure \cite{BK,FS,MR}. 

Let $\ggo$ be a real Lie algebra. A symmetric bilinear form $\la\,,\,\ra$  on $\ggo$ is called {\em ad-invariant} if the following condition
 holds:$$ \la \ad_X Y,Z \ra + \la Y,\ad_X Z\ra = 0 \qquad \mbox{ for all } X, Y, Z \in \ggo.$$
 Whenever $\la\,,\,\ra$ is non-degenerate the symmetric bilinear form is just called a {\em metric}.

\begin{exa} The Killing form is an ad-invariant symmetric bilinear form on any Lie algebra $\ggo$, which is non-degenerate if $\ggo$ is semisimple. Moreover if $\ggo$ is simple any ad-invariant metric on $\ggo$ is a non-zero multiple of the Killing form.
\label{example1}
\end{exa}

Recall that the central descending series $\{C^r(\ggo)\}$ and central ascending series
$\{C_r(\ggo)\}$ of a Lie
algebra $\ggo$, are for $r\geq 0$ respectively given by the ideals 
$$\begin{array}{rclrcl}
C^0(\ggo) & = & \ggo & C_0(\ggo) & = & 0\\
C^r(\ggo) & = & [\ggo, C^{r-1}(\ggo)] & \quad C_r(\ggo) & = &\{X\in \ggo:[X, \ggo]\subseteq C_{r-1}(\ggo)\}.
\end{array}
$$

Fixing a subspace 
$\mm$ of $\ggo$, its orthogonal
subspace
 is defined as usual by
$$\mm^{\perp}=\{ X\in \ggo : \la X, Y\ra=0,\; \forall\; Y \in \mm\}.$$

The next result follows by applying the definitions above and an 
inductive procedure. 

\begin{lem} Let $(\ggo, \la\,,\,\ra)$ denote a Lie algebra endowed with an
ad-invariant metric. 

\begin{enumerate}

\item If $\hh$ is an ideal in $\ggo$ then $\hh^{\perp}$ is also an ideal of
$\ggo$. 

\item $C^r(\ggo)=(C_r(\ggo))^{\bot}$ for all $r\geq 0$.

\end{enumerate}
\label{lema3.1}
\end{lem}

 Notice that if the metric is indefinite, for any subspace $\mm$ the
decomposition $\mm + \mm^{\perp}$ is not necessarily a direct sum. Nevertheless, the next formula holds
\begin{equation}\label{sumco}\dim \ggo = \dim C^r(\ggo)+\dim C_r(\ggo) \quad \forall r\geq 0
\end{equation}
and in particular 
\begin{equation}\label{eq:sumadim}
\dim \ggo =\dim C^1(\ggo) +\dim \zz(\ggo)\end{equation}
 where $\zz(\ggo)$ denotes the center of $\ggo$. Moreover  
 \begin{itemize}
 \item if
 $\mm\subseteq C^1(\ggo)$ is a vector   subspace such that 
 $C^1(\ggo)=(\zz(\ggo)\cap C^1(\ggo))\oplus
 \mm$, then $\mm$ is non-degenerate;
 
 \item if $\mm'\subseteq \zz(\ggo)$ is a vector subspace such that
 $\zz(\ggo)=(\zz(\ggo)\cap C^1(\ggo))\oplus \mm'$, then $\mm'$ is
 non-degenerate.
 \end{itemize}

 \begin{rem} 
 Suppose $\ggo$ admits an ad-invariant metric and $\zz(\ggo)\neq 0$. Then as said above any complementary space $\tilde{\zz}$ such that 
 $\zz(\ggo)=\tilde{\zz}\oplus\left(\zz(\ggo)\cap C^{1}(\ggo)\right)$ is non-degenerate. It follows that  $\ggo=\tilde{\zz}\oplus \tilde{\ggo}$ as a direct sum of non-degenerate ideals where $\tilde{\ggo}=\tilde{\zz}^{\perp}$ each of them having ad-invariant metrics. In addition $\zz(\tilde{\ggo})=\zz(\ggo)\cap C^{1}(\ggo)$. 
 
 Now suppose $\ggo$ is solvable. Then by (\ref{eq:sumadim}) it has non-trivial center. If moreover $\ggo$ is nonabelian then both $C^{1}(\ggo)$ and $\zz(\ggo)$ are non-trivial and $C^1(\ggo)\cap\zz(\ggo)\neq 0$. In fact using the decomposition described above $\ggo=\tilde{\zz}\oplus \tilde{\ggo}$ where $\tilde{\ggo}$ turns to be a solvable Lie algebra with an ad-invariant metric. Then its center $\zz(\tilde{\ggo})=\zz(\ggo)\cap C^{1}(\ggo)$ is not trivial. 
\label{rem3}
 \end{rem}

 \begin{prop} Let $\ggo$ denote a real Lie algebra of dimension two or three. If it can be endowed with an ad-invariant metric, then 
 \begin{itemize}
 \item in dimension two $\ggo$ is  abelian and
 \item in dimension three $\ggo$ is abelian or simple.
 \end{itemize}
 \label{prop1}
 \end{prop}
 \begin{proof}
   Assume first that $\ggo$ has dimension two. Then it is either  abelian or isomorphic to the solvable  Lie algebra spanned by the vectors $X,Y$ with $[X,Y]=Y$. Since the center of this solvable Lie algebra is trivial,  it cannot be equipped with an ad-invariant metric.
 
 Assume now that $\ggo$ has dimension 3. It is well known that it must be either
 solvable or simple. If it is abelian or simple, it admits an ad-invariant metric
 (see Example \ref{example1}). 
 
 Suppose now $\ggo$ is a non-abelian solvable Lie algebra equipped with an ad-invariant bilinear form $\la \,,\,\ra$. Since $\mathfrak{z}(\ggo)\cap C^{1}(\ggo)$ is non-trivial  (see Remark \ref{rem3}), there exist 
  $X,Y\in \ggo$ such that 
  $[X,Y]=Z\in C^1(\ggo) \cap \zz(\ggo)$. 
  It is not difficult to see that the vectors $X,Y,Z$  form a basis of $\ggo$.  Since $Z\in C^{1}(\ggo)\cap \left(C^{1}(\ggo)\right)^{\bot}$ then $\la Z, Z\ra=0$. Furthermore, 
 $$\la Z, X\ra = \la [X,Y], X\ra =-\la Y,[X,X]\ra=0$$
 and in the same way one gets $\la Z, Y\ra=0$. Thus any ad-invariant bilinear form on $\ggo$ must be degenerate.
 \end{proof}
 
  A Lie algebra $(\ggo, \la \,,\,\ra)$ is called {\em indecomposable} if it has no non-degenerate ideals. 
 
 Observe that if a Lie algebra $\ggo$ with an ad-invariant metric  admits a non-degenerate ideal
 $\mathfrak j$, then $\mathfrak j^{\perp}$ is also a non-degenerate ideal and so $\ggo=\mathfrak j \oplus \mathfrak j^{\perp}$. 
 
 \begin{rem} By  Remark \ref{rem3} if $(\ggo, \la \,,\,\ra)$ is indecomposable and with non-trivial center, then  the center is contained in the commutator $\zz(\ggo)\subseteq C^1(\ggo)$. \label{rem2}
 \end{rem}
 
 \begin{lem} \label{led} Let $\ggo$ denote a  Lie algebra of dimension four furnished with an ad-invariant metric. If it is non-solvable then it is decomposable.
\label{solu}
\end{lem}
 
 \begin{proof}
 Let $ \ggo=\mr \oplus \ms$ be a Levi decomposition of $\ggo$, where $\mr$ denotes the radical. Since $\ggo$ is not solvable $\dim \mr <4$. Moreover since there are no simple Lie algebras of dimension one or two, it holds $\dim \mr=1$ and $\ms$ is either $\mathfrak{sl}(2)$ or $\mathfrak{so}(3)$. In every case the action $\ms \to Der(\mr)$ is trivial. In fact let $\mr=\R e_0$ and $\ms=span\{e_1,e_2,e_3\}$. 
 
 Assume $[e_i, e_0]=\lambda_i e_0$. For all $i,j=1,2,3$ there exist $\xi_{ij}\in \R-\{0\}$ such that $[e_i, e_j]=\xi_{ij} e_k$ for some $k=1,2,3$ (see the Lie brackets in $\mathfrak{sl}(2)$ or $\mathfrak{so}(3)$) and where $\xi_{ij}\neq 0$ for all $i,j$. Since $[\ms, \ms]=\ms$  from $\ad([e_i, e_j])e_0=\xi_{ij} \ad(e_k)e_0$ one gets $\lambda_k=0$ for all $k$. 
 
  Let $\la\,,\,\ra$ denote an ad-invariant metric on $\ggo$ and denote $\mu_k=\la e_0, e_k\ra$. So 
  $$\xi_{ij} \mu_k = \xi_{ij}\la e_0, e_k\ra=\la e_0,[e_i, e_j]\ra = \la [e_j, e_0], e_i\ra=0$$
  and since $\xi_{ij}\neq 0$ it must holds $\mu_k=0$ for all $k$.
  Hence since $\la \,,\,\ra$ is non-degenerate, it follows $\la e_0, e_0\ra\neq 0$, so that $\mr$ is a non-degenerate ideal and  the proof is finished.
   \end{proof}

To complete the description of all the Lie algebras of dimension four admitting ad-invariant metrics we have the following result.
 
 \begin{prop}\label{prodim4} Let $\ggo$ denote a real Lie algebra of dimension four which
 can be endowed with an ad-invariant metric. Then
 $\ggo=span\{e_0,e_1,e_2,e_3\}$ is isomorphic to one of the following Lie algebras: 
 \begin{itemize}
 \item $\RR^{4}$
 \item $\RR\oplus \mathfrak{sl}(2,\RR)$
 \item $\RR \oplus \sso(3,\RR)$
 \item the oscillator Lie algebra $\ggo_0=span\{e_{0},\cdots e_{3}\}$ with the non-zero Lie brackets: 
\begin{equation}
[e_0 , e_1 ] = e_2 \quad [e_0 , e_2 ] = -e_1 \quad [e_1 , e_2 ] = e_3
\label{lbg0}
\end{equation}
 \item  $\ggo_1=span\{e_{0},\cdots,e_{3}\}$ with the non-zero Lie brackets: 
  \begin{equation} [e_0 , e_1 ] = e_1 \quad [e_0 , e_2 ] = -e_2 \quad [e_1 , e_2 ] =
 e_3.
 \label{lbg1}
 \end{equation}
 \end{itemize}
 \end{prop}
 
\begin{proof} Let $\ggo$ be a Lie algebra equipped with an ad-invariant metric
$\la\,,\,\ra$.  
If $\ggo$ is decomposable then $\ggo$ corresponds to one of the following Lie algebras: $\RR^{4}$, $\ggo=\RR\oplus \mathfrak{sl}(2,\RR)$, $\ggo=\RR \oplus \sso(3,\RR)$ (by Proposition \ref{prop1}). 

 Assume now $\ggo$ is indecomposable. From Lemma \ref{led} the Lie algebra $\ggo$ is solvable and hence $C^1(\ggo)\neq \ggo$. By Remark \ref{rem2},
$\zz(\ggo)\subseteq C^1(\ggo)$ and $4=\dim \zz(\ggo) +\dim C^1(\ggo) \leq 2 \dim C^1(\ggo)$. It
follows that $\dim \zz(\ggo)=1$ or $\dim\zz(\ggo)=2$. But since we cannot have $\zz(\ggo)=C^{1}(\ggo)$ 
(in dimension four), it should be $\dim\zz(\ggo)=1$ and  $\dim C^1(\ggo)=3$.

Let $e_3$ be a generator of $\zz(\ggo)$ and let
$e_0\in \ggo - C^1(\ggo)$ such that $\la e_0, e_3\ra=1$. 
Denote by $\mm =span\{e_0,e_3\}^{\perp}$. Then $\mm\subseteq \zz(\ggo)^{\bot}=C^1(\ggo)$, $\mm$ is non-degenerate and it is not 
difficult to see that $C^{1}(\ggo)=\zz(\ggo)\oplus\mm$. Then there exists a basis $\{e_1,e_2\}$
of $\mm$ such that the matrix of the metric in this basis takes one of the
following forms
$$B^0=\left( \begin{matrix}
1 & 0 \\ 0 & 1 \end{matrix} \right) \qquad \qquad 
B^{1,1}= \left( \begin{matrix}
0 & 1 \\ 1 & 0 \end{matrix} \right) \qquad \qquad
-B^0=\left( \begin{matrix}
-1 & 0 \\ 0 & -1 \end{matrix} \right).
$$
Thus $C^1(\ggo)=span\{e_1, e_2, e_3\}$ and $e_0$ acts on $C^1(\ggo)$ by the
adjoint action. Due to the ad-invariance property of $\la \,,\,\ra$ it follows that $\ad(e_0)\mm \subseteq \mm$. 

Assume that $\mm$ has the metric given by $B^0$, hence $\ad(e_0) \in \sso(2)$ for $B^0$, implying that
\begin{equation}\label{action1}
 \ad(e_0)=\left( \begin{matrix}
0 & -\lambda \\ \lambda & 0 \end{matrix} \right)
\end{equation}
for some $\lambda\neq 0$. In the case that the metric is given by $-B^0$ the same matrix is obtained for
 $\ad(e_0)$.
Similarly $\ad(e_0) \in \sso(1,1)$ for $B^{1,1}$, implying that
\begin{equation}\label{action2}\ad(e_0)=\left( \begin{matrix}
\lambda & 0 \\ 0 & -\lambda \end{matrix} \right)
\end{equation}
for some $\lambda\neq 0$.

In either case, since $\la [e_0, e_1], e_2\ra= \la e_0, [e_1, e_2]\ra$ one gets
 that $[e_1, e_2]=\lambda e_3$. 

In the basis $\{\frac1{\lambda}e_0, e_1, e_2, \lambda e_3\}$ the action of 
$\ad(\frac1{\lambda}e_0)$ on $\mm$ is as in (\ref{action1}) taking $\lambda=1$ while the metric obeys the rules 
\begin{equation}\label{metric0}
 1 = \la \frac1{\lambda} e_0, \lambda e_3 \ra = \la e_1, e_1 \ra= \la e_2, e_2 \ra \qquad \la e_0, e_0\ra=\mu \in \RR
\end{equation}
and this is for $\ggo_0$. In fact, in this basis the relations of (\ref{lbg0}) are verified.

In the other case a similar reasoning gives the results of the statement, that is, one gets 
the basis $\{e_1, e_2, e_3\}$ for the action (\ref{action2}) and proceeding as above one gets the Lie algebra $\ggo_1$ together with the ad-invariant metric given by:
\begin{equation}\label{metric1}
 1 = \la \frac1{\lambda} e_0, \lambda e_3 \ra = \la e_1, e_2 \ra \qquad \la e_0, e_0\ra=\mu \in \RR.
\end{equation}
\end{proof}

\begin{rem} The ad-invariant metric on the Lie algebra $\ggo_0$ (resp. $\ggo_1$)   can be taken with $\mu=0$. In fact it suffices to change $e_0$ by $\sqrt{\frac{2}{\mu}}e_0-e_3$ whenever $\mu> 0$  and  by $\sqrt{\frac{2}{-\mu}}e_0+e_3$ if $\mu<0$. This gives the following matrices for the ad-invariant metrics
\begin{equation}\label{gmatrix0}
\ggo_0: \quad \left(\begin{matrix}
0 & 0 & 0  & 1\\
0& 1 & 0 & 0\\
0& 0 & 1& 0\\
1& 0& 0 & 0
\end{matrix}\right)  
 \qquad \qquad \ggo_1: \quad\left(\begin{matrix}
0 & 0 & 0  & 1\\
0& 0 & 1 & 0\\
0& 1& 0& 0\\
1& 0& 0 & 0
\end{matrix}\right)
\end{equation}
which will be used from now on. 
\end{rem}

\section{Naturally reductive metrics  on the Heisenberg Lie group}
Let $G$ denote a Lie group with Lie algebra $\ggo$ and let $H< G$ be  a closed Lie subgroup of $G$ whose Lie algebra is denoted by $\hh$. 
A homogeneous pseudo-Riemannian manifold $(M = G/H, \la\,,\,\ra)$ is said to be {\em naturally reductive} if it is reductive, i.e. there is a reductive decomposition
$$\ggo = \hh \oplus \mm \qquad \mbox{ with } \qquad \Ad(H)\mm \subseteq \mm$$
and
$$\la [x, y]_{\mm}, z\ra + \la y, [x, z]_{\mm}\ra = 0 \qquad \mbox{ for all } \quad x, y, z \in \mm.$$

We shall say that a metric on $M$ is naturally reductive if the conditions above are satisfied for some pair $(G,H)$. If $M$ is naturally reductive the geodesics passing through the point $o\in M$ are 
$$\gamma(t)= \exp tx \cdot o \qquad \mbox{ for some } x\in \mm,$$
which implies that these spaces are geodesically complete.
 For the Heisenberg Lie group of dimension 2n+1, $\Heis_{2n+1}(\RR)$,  one has the next result.

\medskip

{\bf Theorem} \cite{Ov2} {\em If $\Heis_{2n+1}(\RR)$ is endowed with a left-invariant pseudo-Rie\-mannian metric for which the center is non-degenerate, then this metric is naturally reductive.  }

\medskip

Our aim here is to characterize the Lorentzian naturally reductive  metrics on the Heisenberg Lie 
group of dimension three. We shall prove a converse of the result above.

\begin{thm} \label{tfin} If $\Heis_3(\RR)$ is endowed with a naturally reductive pseudo-Rieman\-nian left-invariant  metric with pair $(G, \RR)$ where $G$  has 
dimension four and $\RR< G$ acts by isometric automorphisms on $\Heis_3(\RR)$, then the center of $\Heis_3(\RR)$ is non-degenerate.
\end{thm}

Thus the property of the center  being non-degenerate characterizes the naturally reductive
 metrics on $\Heis_3(\RR)$  whenever the isometries fixing a point act by isometric 
isomorphisms.   

As known there is a one-to-one correspondence between left-invariant pseudo-Riemannian metrics on $\Heis_3(\RR)$ and metrics on
 the corresponding Lie algebra $\hh_3$, which is generated by $e_1, e_2, e_3$ obeying the non-trivial Lie bracket relation 
$[e_1, e_2]=e_3$. In order to prove the theorem above we start with the next result, which does not make use of any metric.

\begin{lem} Let $\ggo=\RR e_0 \oplus \hh_3$ where the commutator $C^1(\ggo) \subseteq \hh_3$ and the restriction of $\ad(e_0)$ 
to $\vv=span\{e_1, e_2\}$ is non-singular. If $\mm \subset \ggo$ is a Lie subalgebra of $\ggo$ which is isomorphic to $\hh_3$ 
then $\mm=\hh_3=span\{e_1, e_2, e_3\}$.
\end{lem}
\begin{proof}
Let $\mm$ denote a subalgebra of $\ggo$ such that $\mm=span\{v_1, v_2, v_3\}$ with $[v_1, v_2]=v_3$ and $[v_i, v_3]=0$ for i=1,2. Take
$$v_1= a_0 e_0 + w_1 + a_3 e_3 \qquad v_2 = b_0 e_0 + w_2 + b_3 e_3 \qquad v_3= c_0 e_0 + w_3 + c_3 e_3
$$
where $w_i \in span\{e_1, e_2\}$ for all i=1,2.  Since $C^1(\ggo)\subseteq span\{e_1, e_2, e_3\}$ it follows that $c_0=0$. Let $A$ denote the restriction of $\ad(e_0)$ to $\vv$, thus we have the following equations
$$
\begin{array}{rcl}
v_3 & = & [v_1, v_2] = A(a_0 w_2-b_0 w_1) + \omega(w_1, w_3) e_3\\ 
0 & = & [v_1, v_3]= a_0 A w_3 +\omega(w_1, w_3) e_3\\
0 & = & b_0 A w_3+\omega(w_2, w_3) e_3.
\end{array}
$$
If $a_0$ or $b_0$ is different from zero, then $w_3=0$ and so $v_3=c_3 e_3$. Therefore $a_0 w_2- b_0w_1=0$ and so we can write $w_2$ in terms of $w_1$ or $w_1$ in terms of $w_2$ depending on $a_0\neq 0$ or $b_0\neq 0$ respectively. It is not hard to see that putting these conditions in $v_1, v_2, v_3$ then one gets that the set $v_1, v_2, v_3$ is linearly dependent which is a contradiction. So $a_0=b_0=0$ and $\mm=span\{e_1, e_2, e_3\}$.
\end{proof}

Now if $G$ is a  Lie group acting by isometries on $\Heis_3(\RR)$ which is naturally reductive with pair $(G,H)$,  then
 $G$ is a semidirect extension of $\Heis_3(\RR)$ \cite{CP} and it admits a bi-invariant metric (according to Theorem  2.2 in \cite{Ov1}).
 Hence the Lie algebra of $G$ should be a solvable  Lie algebra of dimension four admitting an
 ad-invariant metric, therefore either  $\ggo_0$ or $\ggo_1$ of the previous section. Thus Theorem \ref{tfin} follows from the next result and the previous lemma.

\begin{lem} Let $\hh_3$ denote the Heisenberg Lie algebra of dimension three equipped with a naturally reductive metric with  
pair  $(\ggo_i, \RR)$ i=0,1 where $\RR\simeq \ggo_i/\hh_3$ acts by skew-adjoint derivations on $\hh_3$. Then the center of $\hh_3$ is non-degenerate.
\end{lem}
\begin{proof} Let $v\in \ggo_i$ be an element which is not in $span\{e_1, e_2, e_3\}$. Thus 
$\ggo_i=\RR v \oplus \hh_3$ and we may assume $v=e_0 + \alpha e_1 + \beta e_2 + \gamma e_3$ and 
$[v, \hh_3] \subseteq \hh_3$. 

For $\ggo_0$ the action of $\ad(v)$ is given by
 $$\ad(v)e_1= e_2 - \beta e_3 \qquad \ad(v)e_2= -e_1 + \alpha e_3 \qquad \ad(v) e_3 \equiv 0.$$
Let $Q$ denote a metric on $\hh_3$ such that $b_{ij}=Q(e_i, e_j)$ and  for which $\ad(v)$ is 
skew-adjoint. The condition  $Q(\ad(v) x, y)=-Q(x, \ad(v) y)$ for all $x,y \in \hh_3$ gives 
rise to  a system of equations on  the coefficients $b_{ij}$: 
$$\begin{array}{rclrclrcl}
b_{12}-\beta b_{13} & = & 0  \qquad b_{22}-\beta b_{13} & = & b_{11}-\alpha b_{13} \qquad b_{23}-\beta b_{33} & = & 0\\
b_{12}-\alpha b_{23} & = & 0  \qquad  b_{13}-\alpha b_{33} & = & 0.  \qquad \qquad \qquad \qquad \qquad  && 
\end{array}
$$
It is not hard to see that if we write $B=(b_{ij})$ then $\det B \neq 0$ implies $b_{33}\neq 0$, 
that is $Q$  non-degenerate implies the center of $\hh_3$ non-degenerate. 
 
This also applies for $\ggo_1$. One writes down the action of $\ad(v)$ and  from 
$Q(\ad(v) x, y)=-Q(x, \ad(v) y)$ the equations follow
$$\begin{array}{rclrclrcl}
b_{11}-\beta b_{13} & = & 0  \qquad b_{12}-\beta b_{23} & = & b_{12}-\alpha b_{13} \qquad b_{13}-\beta b_{33} & = & 0\\
b_{22}-\alpha b_{23} & = & 0   \qquad  b_{23}-\alpha b_{33} & = & 0.\qquad \qquad \qquad \qquad  \qquad & &    
\end{array}
$$
In this case also $b_{33} \neq 0$ says that  the center of $\hh_3$ must be non-degenerate.
\end{proof}
The  simply connected Lie group $\Heis_3(\RR)$  with Lie algebra $\hh_3$ can be
realized on the usual differentiable structure of $\RR^3$ together with the next multiplication
 $$(v,z)\cdot(v',z')=(v+v',z+z'+\frac{1}{2}v^TJv'),$$
where $v, v'\in \RR^2$, $v^T$ denotes the transpose matrix of the 2$\times$1 matrix v, and $J$ denotes
the matrix given by
$$J=\left( \begin{matrix} 0 & 1 \\
-1 & 0 \end{matrix} \right).
$$

A basis of left-invariant vector fields at every point $(x,y,z)\in \RR^3$ satisfying the non-trivial Lie
 bracket relation  $[X_1, X_2]=X_3$ is given by

\smallskip

\qquad $X_1  =   \deax - \frac{y}2\deaz$

\smallskip

\qquad $X_2  =   \deay  + \frac{x}2 \deaz$

\smallskip

\qquad $X_3 =  \deaz.$

\smallskip

Two non-isometric Lorentzian metrics on $\Heis_3(\RR)$ can be taken by defining
\begin{equation}\label{h1}
 1= \la X_1, X_1 \ra = \la X_2, X_2 \ra = - \la X_3, X_3 \ra
\end{equation}
\begin{equation}\label{h2}
 1= \la X_1, X_2 \ra =  \la X_3, X_3 \ra
\end{equation}
and the other relations are zero. Each of them is a naturally reductive  pseudo-Riemannian metric on $\Heis_3(\RR)$ with the following expression  in the usual coordinates of $\RR^3$:
$$\begin{array}{rcl}
   h_1 & =& (1-\frac{y^2}4)dx^2 +(1-\frac{x^2}4)dy^2- dz^2+ \frac14 xy \,dx dy -\frac{y}2 \,dx dz +\frac{x}2\,dy dz\\ \\
h_2 & = & \frac{y^2}4 \,dx^2 + \frac{x^2}4 \, dy^2 +  dz^2+ \frac14 xy \,dx dy +\frac{y}2 \,dx dz - \frac{x}2\,dy dz.
  \end{array}
$$

Making use of this information one can compute several geometrical features on $\Heis_3(\RR)$ \cite{Ov2}. Recall that an {\em algebraic Ricci soliton} on  $\Heis_3(\RR)$ is a left-invariant pseudo-Riemannian metric such that its  Ricci operator $\Rc$ satisfies 
the equality 
$$\Rc(g) = c \, \Id + D\qquad \mbox{ where $c\in \RR$ and $D$ is  a derivation of $\hh_3$,}
$$
 that is $D:\hh_3 \to \hh_3$ is a linear map which satisfies $D[x,y]=[Dx, y] + [x, Dy]$ for all $x,y\in \hh_3$.

\medskip

A pseudo-Riemannian manifold is called {\em locally symmetric} if $\nabla R \equiv 0$, where $\nabla$ denotes 
the covariant derivative with respect to the Levi-Civita connection and $R$ denotes the curvature tensor. The Ambrose-Hicks-Cartan theorem (see for example \cite[Thm. 17, Ch. 8]{ON}) states that given a 
complete locally symmetric 
pseudo-Riemannian  manifold $M$, a linear isomorphism
$A : T_p M \to T_p M$ is the differential of some isometry of $M$ that fixes the point $p\in M$ if and 
only if it preserves the symmetric bilinear form that the metric induces into the tangent space and if for 
every $u,v, w \in T_p M$ the following equation holds:
\begin{equation} \label{ACH}
 R(Au, Av)Aw = A R(u, v)w.
\end{equation}

In \cite{CP} it was proved that the isometry group corresponding to a pseudo-Riemannian left-invariant metric on a 2-step nilpotent Lie algebra is a semidirect product
$\Is(N)= N \rtimes F(N)$, where $F(N)$ denotes the isotropy subgroup at the identity element.
Thus $\Is(N)$ is essentially determined by $F(N)$.
Moreover
\begin{itemize}
\item if $h_0$ is a flat metric on $\Heis_3(\RR)$ then $(\Heis_3(\RR),h_0)$  is a 
locally symmetric space and 
therefore it applies the Ambrose-Hicks-Cartan theorem for the computation of $F(N)$. 
\item for the non-flat metrics the  action of the isotropy subgroup (of the full isometry group) 
at the identity element is given by isometric automorphisms \cite{CP} so
 that  $\Is(\Heis_{2n+1}(\RR)) = \Heis_{2n+1}(\RR)\rtimes H$, where $H$ denotes the group of 
isometric automorphisms. In \cite{Ov2} this group is described.
\end{itemize}

\begin{prop} \label{isoh} The isometry groups for the Lorentzian left-invariant metrics on $\Heis_3(\RR)$ are given by
\begin{itemize}
\item $\Is(\Heis_3(\RR), h_0)= \Heis_3(\RR)\rtimes \Or(2,1)$,
\item $\Is(\Heis_3(\RR), h_1)=  \Heis_3(\RR)\rtimes \Or(2)$,
\item $\Is(\Heis_3(\RR), h_2)= \Heis_3(\RR)\rtimes \Or(1,1)$.
\end{itemize}
Moreover both Lorentzian left-invariant non-flat metrics are algebraic Ricci solitons. 
\end{prop}
\begin{proof}  
The description of the isometry group for a 2-step nilpotent Lie group equipped with a 
left-invariant metric obtained  in \cite{Ov2} and the observations above give the proofs of 
the isometry groups. Notice that the connected component of the identity are $G_0$ and $G_1$ for $h_1$ and $h_2$ respectively (see the description of $G_0$ and $G_1$ in the next section). 

By computing the Ricci tensor in the case of the naturally reductive metrics $h_1$ and $h_2$ one verifies that the corresponding Ricci operators satisfy
 \begin{equation}
 \Rc(h_1)= \Rc(h_2) = \frac32 \Id -D 
 \end{equation}
  where $D$ is the derivation of $\hh_3$ given by 
 $$D(X_1)= -X_1 \qquad D(X_2)= -X_2 \qquad D(X_3)= -2X_3,$$
showing that both $h_1$ and $h_2$ are algebraic Ricci solitons.
\end{proof}

\begin{rem}It can be verified that the Lie groups $G_0$ and $G_1$ act by 
isometries on $(\Heis_3(\RR), h_1)$ and $(\Heis_3(\RR), h_2)$ respectively. 
 Compare with \cite{BR} for the isometry groups. 
For Ricci solitons see \cite{BO}.
\end{rem}
 
\begin{rem}
 A left-invariant Lorentzian metric on $\Heis_3(\RR)$ is flat if and only if the center is degenerate \cite{Ge}. In \cite{Ov3} the flat Lorentzian metric on $\RR \times \Heis_3(\RR)$ given in \cite{Ge} is proved to be  naturally reductive   
and it admits an action by isometries of the free 3-step nilpotent Lie group in two generators. 

Left-invariant pseudo-Riemannian metrics on 2-step nilpotent Lie groups are geodesically complete \cite{Ge0, CP}.
 \end{rem}

\section{Simply connected solvable Lie groups with a bi-invariant metric in dimension four}

Our aim now is to describe geometrical features of the simply connected solvable Lie groups of dimension four provided with a bi-invariant metric. More precisely those corresponding to the Lie algebras $\ggo_0$ and $\ggo_1$ described in Proposition \ref{prodim4}. 

Recall that if $G$ is a connected real Lie group, its Lie algebra $\ggo$ is identified with the Lie algebra of left-invariant vector fields on $G$.  Assume $G$ is 
  endowed with a left-invariant pseudo-Riemannian metric $\la\,,\,\ra$. Then 
  the following statements are equivalent (see \cite[Ch. 11]{ON}):

\begin{enumerate}
\item $\la\, ,\,\ra$  is right-invariant, hence bi-invariant;
\item $\la\,,\,\ra$  is $\Ad(G)$-invariant;
\item  the inversion map $g \to g^{-1}$ is an isometry of $G$;
\item $\la [X, Y], Z\ra + \la Y, [X, Z] \ra= 0$ for all $X,Y,Z \in \ggo$;
\item $\nabla_XY = \frac12 [X, Y]$ for all $X,Y  \in \ggo$, where
$\nabla$ denotes the Levi Civita connection;
\item  the geodesics of $G$ starting at the identity element $e$ are the 
one parameter subgroups of $G$.
\end{enumerate}

 By (3) the pair $(G, \la\,,\,\ra)$ is a pseudo-Riemannian symmetric space. Furthermore by computing
   the curvature tensor one has
 \begin{equation}
 R(X, Y) = - \frac14 \ad([X, Y]) \qquad \quad  \mbox{ for } X,Y \in
 \ggo.
 \label{curvatura}
 \end{equation}
\subsection{Structure of the Lie groups}

 The action of $e_0$ on $\hh_3$ on both  Lie algebras $\ggo_0$ and $\ggo_1$,
lifts to a Lie group homomorphism $\rho: \RR \to \Aut(\Heis_3(\RR))$ which on
 $(v, z)\in \RR^2 \oplus \RR$ has a matrix of the form
\begin{equation}
\rho(t)=\left( \begin{matrix}
R_i(t) & 0 \\
0 & 1 \end{matrix} 
\right) \qquad \qquad i=0,1
\label{rho}\end{equation}
where
\begin{equation}
R_0(t)  =  {\left( \begin{matrix}
\cos \,t & -\sin \,t \\ \sin \,t & \cos \,t \end{matrix} \right) } \mbox{ for }
\ggo_0, \qquad 
R_1(t)  =  {\left( \begin{matrix}
e^t & 0 \\ 0 & e^{-t} \end{matrix} \right)}  \mbox{ for }
\ggo_1.
\label{R0}
\end{equation}

Let $G_0$ and $G_1$ denote the simply connected Lie groups with respective Lie algebras $\ggo_0$ and $\ggo_1$. Then $G_0$ and $G_1$ are modeled on the smooth manifold $\RR^4$, where the algebraic structure is the resulting from the
semidirect product of $\RR$ and $\Heis_3(\RR)$, via $\rho$. Thus on $G_i$ for $i=0,1$, the multiplication is given by
\begin{equation}\label{oper}
 (t,v,z) \cdot (t',v',z')  =  (t+t', v+ R_i(t)v',
z+z'+\frac12 v^T JR_i(t) v'). 
\end{equation}

This information is useful in order to find a basis of the left-invariant
vector fields. For $G_0$ such a basis  at every point $(t,x,y,z)\in \RR^4$ is given by the following vector fields, each of them evaluated at $(t,x,y,z)$: 

\smallskip

\qquad $X_0  =  \deat$

\smallskip

\qquad $X_1  =  \cos\, t \, \deax + \sin \,t \, \deay
 + \frac12(x \,\sin\, t - y \,\cos \,t) \,\deaz$ 

\smallskip

\qquad $X_2  =  -\sin \,t \, \deax + \cos \,t \, \deay + \frac12(x \, \cos\, t + y\, \sin \,t) \, \deaz$

\smallskip

\qquad $X_3 = \deaz$

\smallskip
\noindent and for $G_1$ it is given by

\smallskip

\qquad $X_0  =  \deat$

\smallskip 

\qquad $X_1  =  e^t \, \deax -\frac{1}{2}y\,e^t \, \deaz$
 
 \smallskip
 
 \qquad $X_2  =  e^{-t} \, \deay +\frac{1}{2}x\,e^{-t} \,\deaz$
  
 \smallskip
 
 \qquad $X_3  =  \deaz$.

\smallskip

These vector fields verify the relations given in (\ref{lbg0}) and (\ref{lbg1}) respectively. 
 
For every $i=0,1$ the bi-invariant metric on $G_i$ induced by the ad-invariant metric on
 $\ggo_i$ described in (\ref{gmatrix0}) induces on $\RR^4$  the next pseudo-Riemannian metric (in the usual coordinates):
$$\begin{array}{rcll}
g_0&=&dz\,dt+dx^2+dy^2+\frac12(y dx\,dt-x dy\,dt) & \mbox{ for } G_0\\
 g_1&=&dz\,dt+dx\,dy+\frac12(y dx\,dt-x dy\,dt) & \mbox{ for } G_1.\end{array}
$$

\subsection{Geodesics}  Computing  the Christoffel symbols of the Levi-Civita connection for the metrics $g_0, g_1$ (cf. \cite{ON}),  a curve $\alpha(s)=(t(s),x(s),y(s),z(s))$ is a geodesic in $G_i$ if its components satisfy the  second order system of differential equations:

$\bullet$ for $G_0$
$$\left\{\begin{array}{rcl}
t''(s)&=&0,\\
x''(s)&=&-t'(s)y'(s),\\
y''(s)&=&t'(s)x'(s),\\
z''(s)&=&\frac{1}{2}\;t'(s)(x(s)x'(s)+y(s)y'(s)).
\end{array}\right.$$ 

$\bullet$ for $G_1$
$$\left\{\begin{array}{rcl}
t''(s)&=&0,\\
x''(s)&=&t'(s)x'(s),\\
y''(s)&=&-t'(s)y'(s),\\
z''(s)&=&-\frac{1}{2}\;t'(s)(x(s)y'(s)+y(s)x'(s)).
\end{array}\right.$$ 

On the other hand, if $X_e=\sum_{i=0}^{3}a_{i}X_{i}(e)\in T_{e} G_{0}$, then the geodesic $\alpha$ through $e$ with initial condition $\alpha'(0)=X_e$ is the integral curve of the left-invariant vector field $X=\sum_{i=0}^{3}a_{i}X_{i}$. Suppose $\alpha(s)=(t(s),x(s),y(s),z(s))$ is the curve satisfying  $\alpha'(s)=X_{\alpha(s)}$, then its coordinates are as below.

On $G_0$, for $a_0\neq 0$: 
\begin{eqnarray*} \label{geodcomp}
t(s)&=&a_{0}s,\\
x(s)&=&\frac{a_{1}}{a_{0}}\sin a_{0}s+\frac{a_{2}}{a_{0}}\cos a_{0}s-\frac{a_{2}}{a_{0}},\\
y(s)&=&-\frac{a_{1}}{a_{0}}\cos a_{0}s+\frac{a_{2}}{a_{0}}\sin a_{0}s+\frac{a_{1}}{a_{0}},\\
z(s)&= &\frac{1}{2}\left[ \left(\frac{a_{1}^{2}}{a_{0}}+\frac{a_{2}^{2}}{a_{0}}+2a_{3}\right)s-\left(\frac{a_{2}^{2}}{a_{0}^{2}}+\frac{a_{1}^{2}}{a_{0}^{2}}\right)\sin a_{0}s \right].\end{eqnarray*}

 If $a_{0}=0$, it is easy to see that $\alpha(s)=(0,a_{1}s,a_{2}s,a_{3}s)$ is the corresponding geodesic.

On $G_1$ for $a_0\neq 0$:
\begin{eqnarray*}
t(s)&=&a_{0}s,\\
x(s)&=&\frac{a_{1}}{a_{0}}e^{a_{0}s}-\frac{a_1}{a_0},\\
y(s)&=&-\frac{a_{2}}{a_{0}}e^{-a_0s}+\frac{a_2}{a_0},\\
z(s)&= & \left(\frac{a_{1}a_2}{a_{0}}+a_3\right)s-\frac{a_{1}a_2}{a_{0}^{2}}\sinh(a_0s).
\end{eqnarray*}

 If $a_{0}=0$ again $\alpha(s)=(0,a_{1}s,a_{2}s,a_{3}s)$ is the corresponding geodesic.

As a consequence if $X=\sum_{i=0}^{3}a_{i}X_{i}(e)$, the exponential map is

\noindent
$\bullet$ On $G_0$, if $a_{0}\neq 0$, 
$$\exp (X)=
  \displaystyle{\left(a_0,\frac{1}{a_0}(R_0(a_0)J-J) (a_1,a_2)^t,a_3 +\frac{1}{2}\left(\frac{a_1^2}{a_0}+\frac{a_2^2}{a_0}\right)\left(1-\frac{\sin a_0}{a_0}\right)\right)}$$
 $\text{ if } a_{0}= 0$, 
 $$\exp (X)= \displaystyle{ \left(0,a_1,a_2,a_3\right)}.$$

\noindent
$\bullet$ On $G_1$, if $a_0\neq 0$
$$\exp(X)=\left(a_0,\frac{a_1}{a_0}(e^{a_0}-1),\frac{a_2}{a_0}(1-e^{-a_0}),\frac{a_1a_2}{a_0}+a_3-\frac{a_1a_2}{a_0^2}\sinh(a_0)\right)$$
 $\text{ if } a_{0}= 0$, 
 $$\exp (X)= \displaystyle{ \left(0,a_1,a_2,a_3\right)}.$$

 In both cases the geodesic passing through the point $g\in G_i$, $i=0,1$ and with derivative the left-invariant vector field $X$,  is the translation on the left of the one-parameter group at $e$, that is $\gamma(s)= g \exp(sX)$ for $\exp(sX)$ given above. 
 \subsection{Isometries}
    Let $G$ be a connected Lie group with a bi-invariant metric, and let 
    $\Is(G)$ denote the
isometry group of $G$. This is a Lie group when endowed with the compact-open topology. Let $\varphi$ be an isometry such that $\varphi(e)=x$, for $x\neq e$. Then $L_{x^{-1}} \circ \varphi$ is an isometry which fixes the element $e\in G$. Therefore $\varphi=L_{x} \circ f$ where $f$ is an isometry such that $f(e)=e$.  Let $F(G)$ denote the isotropy subgroup of the identity $e$ of $G$ and let
$L(G) := \{L_g : g \in G\}$, where $L_g$ is the  translation on the left by $g\in G$. 
Then $F(G)$ is a closed subgroup of $G$ and the explanation above says
\begin{equation}\label{desciso}
\Is(G) = L(G) F(G) = \{L_g \circ f : f \in F(G), g\in G\}.\end{equation}
Thus $\Is(G)$ is essentially determined by $F(G)$.

The following lemma is proved by applying the Ambrose-Hicks-Cartan Theorem (\ref{ACH}) to the Lie group $G$ equipped with a bi-invariant metric and whose  curvature formula was given in  (\ref{curvatura}). In this way one gets a geometric proof of the next result (see \cite{Mu}). 
 
 \begin{lem} \label{iso} Let $G$ be a simply connected Lie group with a bi-invariant
 pseudo-Riemannian metric. Then a linear endomorphism $A : \ggo \to \ggo$ 
 is the differential of some isometry in $F(G)$ if and only if for all $X,Y, Z\in \ggo$, the linear map $A$ satisfies the following two conditions:

\vskip 3pt

\ri \quad $\la A X, A Y \ra = \la X, Y\ra $;

\vskip 3pt

\rii \quad $ A[[X, Y], Z] = [[AX, AY], AZ]$.

\vskip 3pt
\end{lem}

Notice that if $G$ is simply connected, every local isometry of $G$ extends to a unique global one. Therefore the full group of isometries of $G$ fixing the identity is isomorphic to the group of linear isometries of $\ggo$ that satisfy condition $\rii$ of Lemma \ref{iso}. 
By applying this to our case, one gets the next result.   

\begin{thm} \label{tiso} Let $G$ be a non-abelian, simply connected solvable Lie group 
of dimension four endowed with a bi-invariant metric. Then the group of
isometries fixing the identity element $F(G)$ is isomorphic to:
     \begin{itemize}
\item $(\{1,-1\}\times \Or(2))\ltimes \RR^2$ for $G_0$,
\item $(\{1,-1\}\times\Or(1, 1))\ltimes \RR^2$ for $G_1$.
\end{itemize}

In particular the connected component of the identity of $F(G)$ coincides with the
group of inner automorphisms $\{I_g:G_0\to G_0,\; I_g(x)=gxg^{-1}\}_{g\in G}$.
\end{thm}
 \begin{proof} We proceed with $\ggo_0$, the case of $\ggo_1$ follows with the same procedure.
 
  Let $A:\ggo_0\to \ggo_0$ be a linear isometry that satisfies the conditions of Lemma \ref{iso}. 
 
 Since $C^1(\ggo_0)$ coincides with $C^2(\ggo_0)$ 
 it follows that $AC^1(\ggo_0)\subseteq C^1(\ggo_0)$. We also have $[C^{1}(\ggo_0),C^{1}(\ggo_0)]=span\{e_3\}$ and from the relation $-Ae_3=[Ae_1,[Ae_1,
 Ae_0]]$ one has $Ae_3=a_{33} e_3$. Thus we may assume that in the basis $\{e_0, e_1, e_2, e_3\}$
 the map $A$ has a matrix of the form
 $$\left( \begin{matrix}
 a_{00} & 0 & 0 & 0\\
 a_{10} & a_{11} & a_{12} & 0\\
 a_{20} & a_{21} & a_{22} & 0\\
 a_{30} & a_{31} & a_{32} & a_{33}
 \end{matrix}
 \right).
 $$
 From $\la Ae_0, Ae_3\ra=1$ it follows that 
 \begin{equation}
 a_{00} a_{33}=1.
 \label{eq1}
 \end{equation}
 From $\la Ae_i,Ae_j\ra=\delta_{ij}$, for $i,j=1,2$ one gets that 
\begin{equation}
\tilde{A} := \left( \begin{matrix}a_{11} & a_{12}\\ a_{21} & a_{22}
 \end{matrix} \right) \in \Or(2).
 \label{eqo2}
 \end{equation}
 Now $A[e_0,[e_1, e_0]]=[Ae_0,[Ae_1, Ae_0]]=Ae_0$ implies
 \begin{equation}
 a_{00}^2a_{11}=a_{11},\qquad\qquad a_{00}^2a_{21}=a_{21}
 \label{eq3}
 \end{equation}
 \quad\mbox{ and }\quad
\begin{equation}
 a_{31}=-a_{00}(a_{10}a_{11} + a_{20}
 a_{21}).
\label{eq4}
 \end{equation}

 Equations (\ref{eq1}), (\ref{eqo2}) and (\ref{eq3}) assert
 \begin{equation}
 a_{00}=a_{33}=\pm 1.
 \end{equation}

Now from $A[e_0,[e_2, e_0]]=[Ae_0,[Ae_2, Ae_0]]=Ae_2$ one has
\begin{equation}
a_{32}=-a_{00}(a_{10}a_{12} + a_{22} a_{20}).
\label{eq5}
\end{equation}
Set $w=(a_{10},a_{20})^{T}$, from (\ref{eq4}) and (\ref{eq5}) it follows that $(a_{31},a_{32})=\mp w^{T}\widetilde{A}$. 

Finally, the relation $\la Ae_0,Ae_0\ra=0$ implies $a_{30}=\mp\frac{1}{2}|| w||^{2}$. 
 Therefore
 \begin{equation}\label{isom}
 A=\left( 
 \begin{matrix}
 \pm 1 & 0 & 0 \\
 w & \tilde{A} & 0 \\
 \mp\frac12 ||w||^2 & \mp w^T \tilde{A} & \pm 1
 \end{matrix}
 \right).
 \end{equation}
where $w\in \RR^2$ and $\tilde{A}\in \Or(2)$. Moreover any matrix of the form (\ref{isom}) verifies (i) and (ii) of Lemma \ref{iso}. This gives a group isomorphic to $(\{1,-1\}\times \Or(2))\ltimes \RR^2$ for which the identity component corresponds to those matrices of the form (\ref{isom}) with $a_{00}=a_{33}=1$ and $\widetilde{A}\in \rm{SO}(2)=$$\{R_0(t):t\in\RR\}$.
 
  On the other hand, the set of isometric automorphisms of $\ggo_0$ coincides with the set $\Ad(G_0)$, that is, the matrices of the form  
  $$\Ad(t,v)= \left( \begin{matrix}
  1 & 0 & 0 \\
  Jv & R_0(t) & 0 \\
  -\frac12 ||v||^2 & -(Jv)^T R_0(t) & 1
  \end{matrix}
  \right), \qquad  v\in \RR^2.
  $$
being $A(t,v)=\Ad(t,v,z)$ for $v=(x,y)$. By dimension and since $Ad(G_0)$ is connected, 
it must coincide with the identity component. 

The procedure for $\ggo_1$ is  the same. In this case we obtain that in the basis 
$\{e_0,\cdots, e_3\}$, the matrix of a linear isometry of $\ggo_1$ that satisfies the conditions
 of Lemma \ref{iso} is of the form
\begin{equation}\label{isomg1}
 A=\left( 
 \begin{matrix}
 \pm 1 & 0 & 0 \\
 w & \tilde{A} & 0 \\
 \mp\frac12 ||w||^2 & \mp w^T \tilde{J}\tilde{A} & \pm 1
 \end{matrix}
 \right).
 \end{equation}
with $w=(x,y)^{T}\in\RR^{2}$, $||w||^2=2xy$, $\tilde{A}\in \mathbf{O}(1,1)$ and $\tilde{J}=\left(\begin{matrix}
 0 & 1 \\
 1& 0 
 \end{matrix}\right)$. 

The matrix  $A(t,v)$ of $\Ad(t,v,z)$ with $v=(x,y)$  is of the form (\ref{isomg1}) with $a_{00}=1$, $w=(-x,y)$ and $\tilde{A}=R_1(t)$.
\end{proof}
\begin{rem} For $G_0$  compare with \cite{Bu}. \end{rem}

\end{document}